\documentclass[11pt]{article}
\usepackage[utf8]{inputenc}
\usepackage{amsmath,amssymb,amsthm}
\usepackage{hyperref}

\newtheorem{theorem}{Theorem}
\newtheorem{lemma}[theorem]{Lemma}

\newtheorem{remark}[theorem]{Remark}

\title{Notes on the Finiteness of Powers of Two with All Even Digits}

\author{Bogdan C. Dumitru
  \thanks{Faculty of Mathematics and Computer Science, University of Bucharest.
    \texttt{bogdan.dumitru@fmi.unibuc.ro}}}
\date{July~2025}

\begin{document}

\maketitle

\begin{abstract}
We study the problem of finding positive integers $n$ such that all the decimal digits of $2^n$ are even, i.e., belong to $\{0, 2, 4, 6, 8\}$.
Computational checks up to $n = 10^{15}$ reveal the known cases $n = 1, 2, 3, 6, 11$ and no additional instances.
We present a self-contained argument, based on a dynamical Borel-Cantelli lemma, that establishes a metric result related to this problem. We show that the set of "initial phases" in a corresponding dynamical system that would generate infinitely many such powers is of Lebesgue measure zero, providing strong probabilistic support for the finiteness conjecture.
\end{abstract}

\vspace{1em}


\section{Introduction}
The problem of finding powers of $2$ whose decimal digits are all even has recently sparked interest in recreational and computational number theory.
While perhaps not a central problem in number theory, the investigation into the finiteness of such integers leads to the application of techniques and tools from metric number theory that are applicable to other digit-based problems.

This paper is structured as follows: in Section~\ref{sec:problem-statement} we reformulate the problem in a more tractable way. Our main theorem and its proof, based on a shrinking-target estimate and the Borel--Cantelli lemma~\cite{Einsiedler2011,Cornfeld1982,Walters2000-lw}, appear in Section~\ref{sec:main}. Then, in Section~\ref{sec:bound} we give a brief heuristic. Section~\ref{sec:comp} records the computational verification. Finally, we conclude in Section~\ref{sec:conclusion}.


\section{Problem Reformulation}
\label{sec:problem-statement}
We approach the problem from the observation that $2^n = 2^{n-1} + 2^{n-1}$ for $n > 1$. For $2^n$ to have only even digits, the addition $2^{n-1} + 2^{n-1}$ must produce no carry. Consequently, all digits of $2^{n-1}$ must be strictly less than $5$. Hence, if $2^n$ has only even digits, then each decimal digit of $2^{n-1}$ lies in $\{0,1,2,3,4\}$ because any digit $\ge 5$ would generate a carry upon doubling.

A crucial step shows that as $d$, the number of digits of $2^k$, increases, the set of fractional parts $f \in [0,1)$ that force the first $d$ digits of $10^f$ to be $\le 4$ becomes so restrictive that its total measure (in $[0,1)$) decays exponentially like $\kappa^d$ for some fixed $\kappa < 1$. We refer to this as being exponentially small in $d$.


\section{Main Result}
\label{sec:main}

Before presenting the main result, we recall a key lemma based on the Borel-Cantelli principle adapted for dynamical systems \cite{Chernov2001}:

\begin{lemma}
\label{lemma:bc}
Let $T: X \to X$ be a measure-preserving transformation on a probability space $(X, m)$, where $m$ is a probability measure. If $(A_n)_{n=1}^\infty$ is a sequence of measurable sets satisfying $\sum_{n=1}^\infty m(A_n) < \infty$, then for $m$-almost every $x \in X$, the set $\{ n \ge 1 : T^n x \in A_n \}$ is finite.
\end{lemma}
Now, let $2^k = 10^{\lfloor kc\rfloor}\cdot 10^{\{kc\}}$ so the base-10 digits of $2^k$ coincide with the first $d(k)=\lfloor kc\rfloor+1$ digits of $10^{\{kc\}}$ (up to a null set of endpoints). The main claim is a metric result concerning the reformulated problem (all digits of a number are $\le 4$).

\begin{theorem}
\label{thm:main}
Let $c = \log_{10}(2)$. Let $d(k) = \lfloor k c \rfloor + 1$ and define, for each integer $d \ge 1$, the set $B_d \subset [0, 1)$ as:
\[
  B_d = \left\{ f \in [0,1) \;\middle|\;
    \begin{gathered} \text{the first } d \text{ digits of } 10^f \text{ are all in } \{0,1,2,3,4\}, \\
    \text{and the leading digit is in } \{1,2,3,4\} \end{gathered}
    \right\}.
\]
The set of starting points $x_0 \in [0,1)$ for which the relation $\{x_0 + k c\} \in B_{d(k)}$ holds for infinitely many integers $k$ has Lebesgue measure zero.
\end{theorem}

\begin{proof}
We connect the condition on the digits of $2^k$ to the properties of the fractional part $\{k c\}$, define the target sets $A_k = B_{d(k)}$, estimate their measure, show this measure is summable, and apply Lemma \ref{lemma:bc}.

The condition "all digits of $2^k$ are $\le 4$" is precisely equivalent to the fractional part $\{k c\}$ being an element of the set $B_{d(k)}$. The number of decimal digits of $2^k$ is $d(k) = \lfloor k c \rfloor + 1$. The sequence of digits of $2^k$ is determined by the value $10^{\{k c\}}$, which is in the interval $[1, 10)$. The set $B_{d(k)}$ represents the "allowed" values for $\{kc\}$.

\vspace{1em}
\noindent
\textbf{Bound on the Measure of $B_d$}
\\
We estimate the Lebesgue measure, $m(B_d)$. Let $y = 10^f$. The condition that the first $d$ digits of $y$ are in $\{0,1,2,3,4\}$ (with the first in $\{1,2,3,4\}$) restricts $y$ to a set of real numbers $S_d \subset [1, 5)$. The total Lebesgue length of $S_d$ is given by the sum of the lengths of the allowed intervals: $\text{length}(S_d) = 4 \cdot (5/10)^{d-1} = 4 \cdot (1/2)^{d-1} = 2^{3-d}$. Converting this length back to a measure in $f$-space using $df = \frac{dy}{y \ln(10)}$ and the bound $1/y \le 1$ for $y \in S_d$:
\[
 m(B_d) = \int_{S_d} \frac{dy}{y \ln(10)} \le \frac{\text{length}(S_d)}{\ln(10)} = \frac{2^{3-d}}{\ln(10)}.
\]
This shows that $m(B_d)$ decays exponentially as $d$ increases.

\vspace{1em}
\noindent
\textbf{Summability of Measures}
\\
Let $A_k = B_{d(k)}$. We check if $\sum_{k=1}^\infty m(A_k)$ converges. The number of $k$ values for which $d(k)$ equals a specific $d$ is bounded by $M = \lceil 1/c \rceil = 4$. (Since $d(k)=d \iff d-1 \le kc < d \iff \frac{d-1}{c} \le k < \frac{d}{c}$, an interval of length $1/c \approx 3.32$).
Thus,
\[
\sum_{k=1}^\infty m(A_k) = \sum_{d=1}^\infty \sum_{k: d(k)=d} m(B_d) \le M \sum_{d=1}^\infty m(B_d).
\]
Since $m(B_d)$ decays exponentially (like $(1/2)^d$), $\sum m(B_d)$ converges rapidly. Consequently, $\sum_{k=1}^\infty m(A_k)$ converges.

\vspace{1em}
\noindent
\textbf{Dynamical Argument}
\\
Consider the dynamical system $(X, \mathcal{B}, m, T)$ where $X=[0,1)$, $\mathcal{B}$ is the Borel $\sigma$-algebra, $m$ is the Lebesgue measure, and $T(x) = \{x+c\}$ is the irrational rotation by $c = \log_{10}(2)$.

We established that the sequence of target sets $A_k = B_{d(k)}$ satisfies $\sum_{k=1}^\infty m(A_k) < \infty$. The dynamical Borel-Cantelli lemma (Lemma \ref{lemma:bc}) directly implies that for $m$-almost every $x_0 \in X$, the set $\{ k \ge 1 : T^k x_0 \in A_k \}$ is finite. This is precisely the statement of the theorem.
\end{proof}

\begin{remark}
Theorem~\ref{thm:main} is a metric statement about starting points $x_0\in[0,1)$ for the rotation $x\mapsto \{x+c\}$. The original problem corresponds to the single orbit starting at $x_0=0$. The theorem shows that the set of $x_0$ leading to infinitely many hits is null, so for the specific rotation by $c=\log_{10}2$ one would need $x_0=0$ to lie in this null set for infinitely many solutions to occur. Establishing (or excluding) such exceptional behavior for this specific orbit is beyond the scope of our method, so the finiteness question remains open.
\end{remark}


\section{Heuristic bound}
\label{sec:bound}

While the previous section established a rigorous metric result, we can gain intuition for the original conjecture using a non-rigorous probabilistic heuristic. Assume the digits of $2^k$ behave like independent random digits chosen uniformly from $0$ to $9$. Since the probability of a single random digit being $\le 4$ is $p=5/10=1/2$, the probability that all $d$ digits of $2^k$ satisfy this condition is approximately $p^d = (1/2)^d$. We know the number of digits is $d(k) = \lfloor k c \rfloor + 1 \approx k c$, where $c = \log_{10}(2) \approx 0.3010$. Thus, the probability for a given $k$ is roughly $(1/2)^{kc} = (2^{-c})^k \approx (0.812)^k$. The expected number of positive integers $k$ satisfying the condition can be estimated by summing these probabilities (treating them as independent events):
\[
 E[\# \text{solutions } k\ge 1] \approx \sum_{k=1}^\infty (0.812)^k = \frac{0.812}{1 - 0.812} \approx 4.32.
\]
This small expected value aligns with the observation that only a few small integers are known to satisfy the condition (namely, $k=0, 1, 2, 5, 10$ for digits of $2^k \le 4$, corresponding to $n=1, 2, 3, 6, 11$ for the original even-digit problem). Although this calculation relies on unproven assumptions about digit randomness, it suggests solutions are rare, consistent with our metric finiteness proof.


\section{Computational verification}
\label{sec:comp}

We used the no-carry reformulation to drive a fast sieve. If $2^n$ has only even digits, then $2^{n-1}$ has all digits in $\{0,1,2,3,4\}$. In particular, the \emph{last $D$ digits} of $2^{n-1}$ must lie in $\{0,1,2,3,4\}$. Fix $D=54$ and let $M=10^{D}$. Writing $p=n-1$, we maintain the residue $r_p \equiv 2^p \pmod{M}$ via the recurrence $r_{p+1}\equiv 2r_p \pmod{M}$ and, at each step, check whether all $D$ base-10 digits of $r_p$ are $\le 4$. This is a necessary condition for the original property, hence it produces only false positives (never false negatives). Whenever the sieve were to pass, we would compute $2^n$ exactly and verify the full condition.

The implementation is a simple C/GMP/OpenMP program. We represent $r_p$ as three 64-bit limbs, update by a single modular multiply per step (no GMP calls in the inner loop), and parallelize by assigning each thread an arithmetic progression in the exponent with stride equal to the number of threads (precomputing $2^{\text{stride}}\bmod M$). The exponent range was partitioned into contiguous batches $[P_{\mathrm{start}},P_{\mathrm{end}}]$ with the next batch beginning at $P_{\mathrm{end}}{+}1$. Scanning $1\le n\le 10^{15}$ (equivalently $0\le p\le 10^{15}\!-\!1$), the sieve identified only the known small cases $n=1,2,3,6,11$; in particular, there were no sieve hits for $10^{13}\le n\le 10^{15}$. Since any genuine solution must pass the sieve, this rules out additional solutions in that range.


\section{Conclusion}
\label{sec:conclusion}

In summary, we established a strong metric result concerning the problem of powers of two with all even digits. By rephrasing the no-carry condition as a digit constraint and applying a dynamical Borel–Cantelli argument, we proved that the set of initial conditions in a corresponding dynamical system that would lead to infinitely many such numbers is of Lebesgue measure zero. While this does not resolve the conjecture for the specific case of base 2, it demonstrates that any potential counterexample must be exceptionally rare, depending on the deep Diophantine properties of $\log_{10}(2)$. Our method rigorously frames the problem in the context of shrinking targets and confirms the scarcity of solutions. Empirical checks confirm that for $n \le 10^{15}$, the only solutions are $n = 1, 2, 3, 6, 11$. 

\newpage

\bibliographystyle{plain}
\bibliography{p2e_11}

\begin{thebibliography}{1}

\bibitem{Chernov2001}
N.~Chernov and D.~Kleinbock.
\newblock Dynamical borel-cantelli lemmas for gibbs measures.
\newblock {\em Israel Journal of Mathematics}, 122(1):1–27, December 2001.

\bibitem{Cornfeld1982}
I.~P. Cornfeld, S.~V. Fomin, and Ya.~G. Sinai.
\newblock {\em Ergodic Theory}.
\newblock Springer New York, 1982.

\bibitem{Einsiedler2011}
Manfred Einsiedler and Thomas Ward.
\newblock {\em Ergodic Theory: with a view towards Number Theory}.
\newblock Springer London, 2011.

\bibitem{Walters2000-lw}
Peter Walters.
\newblock {\em An Introduction to Ergodic Theory}.
\newblock Graduate texts in mathematics. Springer, New York, NY, October 2000.

\end{thebibliography}

\end{document}